\documentclass[12pt,reqno]{amsart}
\usepackage{amsmath,amssymb,verbatim,comment,color,url}
\usepackage[pagebackref,pdftex,hyperindex]{hyperref}

\usepackage[margin=3cm]{geometry}
\usepackage{color}

\newcommand{\C}{\mathbb{C}}

\newcommand{\N}{\mathbb{N}}

\newcommand{\R}{\mathbb{R}}

\newcommand{\cE}{\mathcal{E}}

\newcommand{\cH}{\mathcal{H}}
\newcommand{\cI}{\mathcal{I}}

\newcommand{\cK}{\mathcal{K}}

\newcommand{\cM}{\mathcal{M}}

\renewcommand{\a}{\alpha}
\renewcommand{\b}{\beta}

\renewcommand{\d}{\delta}

\newcommand{\bphi}{\boldsymbol \phi}
\newcommand{\bvarphi}{\boldsymbol \varphi}
\newcommand{\bpsi}{\boldsymbol \psi}
\newcommand{\bomega}{\boldsymbol \omega}

\newcommand{\bp}{\bar{\partial}}
\newcommand{\dd}{\sqrt{-1}\partial \bar{\partial}}

\newcommand{\cf}{{\rm cf.\ }} 
 
\newcommand{\ie}{{\rm i.e.\ }} 

\DeclareMathOperator{\AM}{AM}
\DeclareMathOperator{\Aut}{Aut}

\DeclareMathOperator{\Id}{Id}
\DeclareMathOperator{\Iso}{Iso}
\DeclareMathOperator{\Lie}{Lie}

\DeclareMathOperator{\PSH}{PSH}
\DeclareMathOperator{\osc}{osc}

\DeclareMathOperator{\Ric}{Ric}

\DeclareMathOperator{\Tr}{Tr}

\renewcommand{\leq}{\leqslant}
\renewcommand{\geq}{\geqslant}

\renewcommand{\hat}{\widehat}
\renewcommand{\tilde}{\widetilde}

\numberwithin{equation}{section}       

\newtheorem{prop} {Proposition} [section]
\newtheorem{thm}[prop] {Theorem} 
\newtheorem{dfn}[prop] {Definition}
\newtheorem{lem}[prop] {Lemma}

\newtheorem{rem}[prop]{Remark}
\theoremstyle{remark}
\newtheorem*{ackn}{\bf{Acknowledgment}}

\title[Ricci iteration for coupled K\"ahler-Einstein metrics]{Ricci iteration for coupled K\"ahler-Einstein metrics} 
\date{\today}

\author[R. Takahashi]{Ryosuke Takahashi}
\address{Research Institute for Mathematical Sciences\\
 Kyoto University\\
 Kyoto 606-8502\\
 JAPAN}
\email{tryosuke@kurims.kyoto-u.ac.jp}

\begin{document}

\maketitle
\begin{abstract}
In this paper, we introduce the ``coupled Ricci iteration'', a dynamical system related to the Ricci operator and twisted K\"ahler-Einstein metrics as an approach to the study of coupled K\"ahler-Einstein (CKE) metrics. For negative first Chern class, we prove the smooth convergence of the iteration. For positive first Chern class, we also provide a notion of coercivity of the Ding functional, and show its equivalence to existence of CKE metrics. As an application, we prove the smooth convergence of the iteration on CKE Fano manifolds assuming that the automorphism group is discrete.
\end{abstract}
\section{Introduction}
Let $X$ be an $n$-dimensional compact K\"ahler manifold with $\lambda c_1(X)>0$ ($\lambda=\pm 1$). The problem of finding canonical metrics is one of the most active topics in K\"ahler geometry. Especially we will focus on the ``coupled K\"ahler-Einstein metrics'' introduced by Hultgren-Witt Nystr\"om \cite{HN17}: for any integer $k \in \N$, a $k$-tuple of K\"ahler classes $(\Omega_i)$ is called a {\it decomposition of $2 \pi \lambda c_1(X)$} if
\[
2 \pi c_1(X)=\lambda \sum_{i=1}^k \Omega_i.
\]
Then we say that a $k$-tuple of K\"ahler metrics $\bomega=(\omega_i)$ ($\omega_i \in \Omega_i$) is {\it coupled K\"ahler-Einstein} (CKE) if it satisfies the equations
\begin{equation} \label{CKE}
\Ric(\omega_1)= \ldots =\Ric(\omega_k)=\lambda \sum_{i=1}^k \omega_i.
\end{equation}
Not that if $k=1$, then $\omega_1$ is CKE if and only if $\omega_1$ is K\"ahler-Einstein. Since this is a system of elliptic partial differential equations, several classical methods are available to construct solutions. Indeed, Hultgren and Witt-Nystr\"om \cite{HN17} demonstrated that a variational method can be used to prove existence of CKE metrics. On the other hand, another common method, the continuity method was studied by Pingali \cite{Pin18}. Examples of CKE metrics were studied on toric Fano manifolds by Hultgren \cite{Hul17}, and Delcroix-Hultgren \cite{DH18} in more general settings. Sasakian analogues were also studied by Futaki-Zhang \cite{FZ18}. Moreover, a moment map interpretation for CKE metrics was established very recently by Datar-Pingali \cite{DP19}. 

In this paper, we study a third approach, drawing some of its motivation from dynamical systems. For the moment, for any K\"ahler form $\a$, we consider the {\it $\a$-twisted K\"ahler-Einstein (TKE) metric}
\[
\Ric(\omega)=\lambda(\omega+\a),
\]
which is also a variant of K\"ahler-Einstein metrics, arises in various settings and have been discussed exhaustively (for instance, see \cite{CS16, Fin04, ST09}). The system of equations \eqref{CKE} is equivalent to say that {\it each $\omega_i$ is the $\sum_{j \neq i} \omega_j$-TKE metric for $i=1,\ldots, k$}. In order to gain more geometrical interpretations, it is interesting to study relations between the both metrics using a dynamical system: for any given $k$-tuple of K\"ahler metrics $\bomega_0:=(\omega_{0,i})$ with $\omega_{0,i} \in \Omega_i$, we define $\bomega_\ell:=(\omega_{\ell,i})$ ($\omega_{\ell,i} \in \Omega_i$, $\ell \geq 0; i=1,\ldots, k$) according to the following evolution rule\footnote{We can consider the system \eqref{CRI} for $\lambda=0$. However, in this case, the solution to the first step of \eqref{CRI} is just the $k$-tuple of Calabi-Yau (Ricci-flat) metrics \cite{Yau78}. So the iteration is accomplished in one step.}:
\begin{equation} \label{CRI}
\bomega_{\ell+1}=
\begin{pmatrix}
\omega_{\ell+1,1} \\
\vdots \\
\omega_{\ell+1,i} \\
\vdots \\
\omega_{\ell+1,k}
\end{pmatrix}
:=\cI(\bomega_\ell)
:=
\begin{pmatrix}
(\Ric-\lambda \Id)^{-1} \big[ \lambda \sum_{j>1} \omega_{\ell,j} \big] \\
\vdots \\
(\Ric-\lambda \Id)^{-1} \big[ \lambda \big(\sum_{j<i} \omega_{\ell+1, j}+\sum_{j>i} \omega_{\ell,j} \big) \big] \\
\vdots \\
(\Ric-\lambda \Id)^{-1} \big[ \lambda \sum_{j<k} \omega_{\ell+1, j} \big]
\end{pmatrix}.
\end{equation}
So the K\"ahler metric $\omega_{\ell+1,i}$ is determined as the $(\sum_{j<i} \omega_{\ell+1, j}+\sum_{j>i} \omega_{\ell,j})$-TKE metric, and the motivation for the system \eqref{CRI} is to construct a solution to the CKE equation (coupled Monge-Amp\`ere equations) by repeatedly solving the TKE equations (single Monge-Amp\`ere equation). Another, perhaps more straightforward iteration is given by
\[
\omega_{\ell+1.i}=(\Ric-\lambda \Id)^{-1} \bigg[ \lambda \sum_{j \neq i} \omega_{\ell,j} \bigg].
\]
However, we will see later that \eqref{CRI} fits better into the variational framework for CKE metrics. The well-definedness of the twisted inverse Ricci operator $(\Ric-\lambda \Id)^{-1}$ in each step is guaranteed provided the $k$-tuple $(\Omega_i)$ admits CKE metrics (see Lemma \ref{TKE} for more details). For instance, when $k=2$, the iteration process is described by
\begin{align*}
\text{Step 1}: \Ric(\omega_{1,1})=\lambda(\omega_{1,1}+\omega_{0,2}), \quad \Ric(\omega_{1,2})=\lambda(\omega_{1,1}+\omega_{1,2}), \\
\text{Step 2}: \Ric(\omega_{2,1})=\lambda(\omega_{2,1}+\omega_{1,2}), \quad \Ric(\omega_{2,2})=\lambda(\omega_{2,1}+\omega_{2,2}), \\
\text{Step 3}: \Ric(\omega_{3,1})=\lambda(\omega_{3,1}+\omega_{2,2}), \quad \Ric(\omega_{3,2})=\lambda(\omega_{3,1}+\omega_{3,2}), \\
\vdots \hspace{51mm}
\end{align*}
We would like to call \eqref{CRI} the {\it coupled Ricci iteration}. One can easily check that stationary points of the iteration are exactly CKE metrics. Moreover, the iteration \eqref{CRI} has a monotonicity property with respect to a certain energy functional on the space of K\"ahler metrics, which gives an intuition as to why this method works well (\cf Proposition \ref{gpi}). The inverse Ricci operator as well as some twisted inverse Ricci operators were introduced by Rubinstein \cite{Rub08} and as a generalization to the system \eqref{CRI}. More precisely, he introduced some dynamical systems, so called the ``time $\tau$ Ricci iteration'' as discretization of the K\"ahler-Ricci flow, which was further studied in \cite{BBEGZ11,DR19}. Motivated by Rubinstein's work, we would like to iterate the twisted inverse Ricci operator. We emphasize that
the system \eqref{CRI} does not reduce to the Rubinstein's one when $k=1$ since in this case it reduces to
\[
\omega_{1,1}=(\Ric-\lambda \Id)^{-1}(0),
\]
and the process is accomplished in one step.

In \cite{HN17}, the authors prove that when $\lambda=-1$, any decomposition of $-2\pi c_1(X)$ admits a unique CKE metric. So in this case, we can show convergence of the coupled Ricci iteration without any extra assumptions:
\begin{thm} \label{NFC}
Let $X$ be a compact K\"ahler manifold with $c_1(X)<0$, $(\Omega_i)$ a decomposition of $-2 \pi c_1(X)$ and $\bomega_0=(\omega_{0,i})$ a $k$-tuple of K\"ahler metrics with $\omega_{0,i} \in \Omega_i$. Then the coupled Ricci iteration $\bomega_\ell$ \eqref{CRI} starting from $\bomega_0$ converges to the unique coupled K\"ahler-Einstein metric as $\ell \to \infty$ in the $C^{\infty}$-topology.
\end{thm}
In Theorem \ref{NFC}, we used the uniqueness theorem \cite[Theorem 1.3]{HN17}, while Theorem \ref{NFC} gives an alternative proof for existence based on the continuity (\cf Lemma \ref{continuity of the iteration map}) and strict monotonicity of the iteration map $\cI$ (\cf Proposition \ref{gpi}).

However, in the $\lambda=1$ case, the existence of CKE metrics is more subtle, and a certain algebraic stability condition for CKE, generalizing K-polystability was introduced in \cite{HN17}. However, from the analytic point of view, existence of canonical K\"ahler metrics should have an equivalent formulation in terms of coercivity of certain energy functionals. Thus we study this direction with the aid of the general framework developed by Darvas-Rubinstein \cite{DR15}, and obtain the following:
\begin{thm} \label{Coercivity}
Let $X$ be a compact K\"ahler manifold with $c_1(X)>0$ and $(\Omega_i)$ be a decomposition of $2 \pi c_1(X)$. Then $(X, (\Omega_i))$ admits a coupled K\"ahler-Einstein metric if and only if the Ding functional ${\bf D}$ is ${\bf J}_G$-coercive.
\end{thm}
These functionals are essentially just the summation of the corresponding functionals for the $k$ classes $(\Omega_i)$.
For the precise definition of functionals and coercivity, see Section \ref{FEP}.
As an application, we also prove the smooth convergence of the coupled Ricci iteration under restrictive assumptions on the automorphism group.
\begin{thm} \label{PFC}
Let $X$ be a compact K\"ahler manifold with $c_1(X)>0$ and discrete automorphism group, $(\Omega_i)$ a decomposition of $2 \pi c_1(X)$ admitting a coupled K\"ahler-Einstein metric. Then for any $k$-tuple of K\"ahler metrics $\bomega_0=(\omega_{0,i})$ with $\omega_{0,i} \in \Omega_i$, the coupled Ricci iteration $\bomega_\ell$ \eqref{CRI} starting from $\bomega_0$ converges to the coupled K\"ahler-Einstein metric as $\ell \to \infty$ in the $C^\infty$-topology.
\end{thm}
Fortunately, one has a uniform positive Ricci lower bound along the iteration, and can use a Green function argument essentially due to Bando-Mabuchi \cite{BM85} for the $C^0$-estimates. As is the case with the time $\tau$ Ricci iteration \cite{DR19}, we expect that the assumption on $\Aut_0(X)$ can be removed and convergence of the coupled Ricci iteration holds in the $C^{\infty}$ Cheeger-Gromov sense.

This paper is organized as follows. In Section \ref{PLS}, we recall some basic results for the $L^1$-Finsler geometry and functionals on the space of PSH potentials using the framework by Darvas-Rubinstein \cite{DR15}. We introduce ${\bf J}_G$-coercivity and prove Theorem \ref{Coercivity}. See \cite{HN17, Rub14} for a standard exposition on (coupled) K\"ahler-Einstein metrics and \cite{Dar19,Rub18} for the exposition on Tian's properness conjectures, which both give a survey of \cite{DR15}. We also discuss a Matsushima type decomposition and the well-definedness of the twisted inverse Ricci operator $(\Ric-\lambda \Id)^{-1}$. In Section \ref{CIB}, we study some crucial properties for the iteration map $\cI$ (\cf Lemma \ref{continuity of the iteration map}, Proposition \ref{gpi}). Then we prove smooth convergence of the coupled Ricci iteration in the two cases $\lambda=\pm 1$ separately.
\begin{ackn}
The author expresses his gratitude to Prof. Shigetoshi Bando for helpful conversations and pointing out the formula \eqref{isometry}. The author was supported by Grant-in-Aid for JSPS Fellows Number 16J01211.
\end{ackn}
\section{Preliminaries} \label{PLS}
\subsection{$L^1$-Finsler geometry and energy functionals on the space of PSH potentials} \label{FEP}
Let $X$ be a compact K\"ahler manifold with $\lambda c_1(X)>0$, $\omega_0$ a K\"ahler metrics in a K\"ahler class $\Omega$. First, we set
\[
\cH:=\{ \phi \in C^{\infty}(X;\R)|\omega_{\phi}:=\omega_0+\dd \phi>0\}.
\]
For a smooth curve $\a \colon [0,1] \to \cH$, the length of $\a$ is defined by
\[
\ell_1(\a):=\int_0^1 \int_X |\dot{\a}(t)| \omega_{\alpha(t)}^n dt.
\]
Then for any $\phi, \psi \in \cH$, we define the $L^1$-Finsler metric by
\[
d_1(\phi,\psi):=\inf \{ \ell_1(\a)| \text{$\a \colon [0,1] \to \cH$ is a smooth curve with $\a(0)=\phi$ and $\a(1)=\psi$} \}.
\]
The {\it Aubin-Mabuchi energy} is defined by
\[
\AM(\phi):=\frac{1}{n+1} \sum_{j=0}^n \int_X \phi \omega_{\phi}^j \wedge \omega_0^{n-j}.
\]
Set also
\[
\tilde{\cH}:=\cH \cap \AM^{-1}(0).
\]
Then the subspace $\tilde{\cH}$ is isomorphic to
\[
\cK:=\{\omega_{\phi} \in [\omega_0]|\phi \in \cH\}
\]
via the natural correspondence $\phi \mapsto \omega_{\phi}$. We use this isomorphism to endow $\cK$ with a metric structure, by pushing forward the $L^1$-Finsler metric on $\cH$. The Aubin-Mabuchi energy has an extension to $\PSH(X, \omega_0)$, the space of $\omega_0$-PSH functions given by
\[
\AM(\phi):=\inf \{ \AM(\psi) |\psi \in \PSH(X, \omega_0), \text{$\psi$ is smooth}, \psi \geq \phi \}.
\]
Then the finite energy class $\cE^1$ is defined by
\[
\cE^1:=\{ \phi \in \PSH(X, \omega_0) | \AM(\phi)>-\infty \}.
\]
The Aubin-Mabuchi energy satisfies the scaling property $\AM(\phi+c)=\AM(\phi)+[\omega_0]^n \cdot c$ for all $\phi \in \cE^1$ and $c \in \R$. The space $\cE^1$ also arises as the completion of the metric space $(\cH, d_1)$ \cite[Lemma 5.2]{DR15}. We define the functionals $I$, $J$ on $\cE^1$ as follows
\[
J(\phi):=-\frac{1}{[\Omega]^n} \AM(\phi)+\frac{1}{[\Omega]^n} \int_X \phi \omega_0^n,
\]
\[
I(\phi):=\frac{1}{[\Omega]^n} \int_X \phi(\omega_0^n-\omega_{\phi}^n),
\]
where $\omega_{\phi}^n$ means the non-pluripolar product (\cf \cite{BEGZ10}). These functionals are comparable to each other as follows (\cf \cite[Section 2.1]{BBGZ13})
\begin{equation} \label{JAI}
\frac{1}{n+1} I(\phi) \leq J(\phi) \leq I(\phi).
\end{equation}
Also under the $\AM$-normalization, $d_1(0,\cdot)$ is equivalent to $J$ (\cf \cite[Lemma 5.11]{DR15}):
\begin{equation} \label{Jde}
\frac{1}{C} J(\phi)-C \leq d_1(0,\phi) \leq C J(\phi)+C, \quad \phi \in \cE^1 \cap \AM^{-1}(0).
\end{equation}

Next we work in the CKE setting. Let $(\Omega_i)$ be a decomposition of $2 \pi \lambda c_1(X)$ and $\bomega_0=(\omega_{0,i})$ a $k$-tuple of K\"ahler metrics with $\omega_{0,i} \in \Omega_i$. We take also a K\"ahler form $\theta_0$\footnote{The definition of $\theta_0$ is used only for a Green function estimate in the proof of Lemma \ref{unc}.} satisfying the Calabi-Yau equation \cite{Yau78}
\[
\Ric(\theta_0)=\lambda \sum_i \omega_{0,i}.
\]
Applying the above constructions to each variable $i$, we obtain the corresponding spaces and functionals ($\cH_i$, $\cE^1_i$, $\AM_i$, etc). For $\bphi=(\phi_i) \in \prod_i \cE^1_i$, we define
\[
{\bf J}(\bphi):=\sum_{i=1}^k J_i(\phi_i),
\]
\[
{\bf L}(\bphi):=-\lambda \log \int_X e^{-\lambda \sum_i \phi_i} \theta_0^n.
\]
Then the {\it Ding functional} introduced in \cite{HN17} is given by
\[
{\bf D}(\bphi):=-\sum_{i=1}^k \frac{1}{[\Omega_i]^n} \AM_i(\phi_i)+{\bf L}(\bphi).
\]
Also let $\rho_i(\bphi) \in C^\infty(X;\R)$ ($i=1,\ldots,k$) be the {\it coupled Ricci potential} uniquely determined by
\[
\Ric(\omega_{\phi_i})-\sum_j \omega_{\phi_j}=\dd \rho_i(\bphi), \quad \int_X e^{\rho_i(\bphi)} \omega_{\phi_i}^n=[\Omega_i]^n.
\]
Then a direct computation shows that
\[
\d{\bf D}|_{\bphi}(\d \bphi)=-\sum_i \frac{1}{[\Omega_i]^n} \int_X \d \phi_i(1-e^{\rho_i(\bphi)})\omega_{\phi_i}^n.
\]
So the Euler-Lagrange equation of ${\bf D}$ is given by
\[
\rho_1(\bphi)=\ldots=\rho_k(\bphi)=0,
\]
which is equivalent to say that $\bphi$ is CKE.

Finally, we take account of the action of automorphisms $\Aut(X)$. Set $G:=\Aut_0(X)$, the identity component of the automorphism group $\Aut(X)$. Then $g \in G$ acts on $\cK$ by pullback. Given one-to-one correspondence between $\cK$ and $\cH$, the group $G$ also acts on $\cH$, \ie for any $g \in G$ and $\phi \in \cH$, $g. \phi$ is defined as the unique element in $\tilde{\cH}$ such that $g^\ast \omega_\phi=\omega_{g.\phi}$. Then the $G$-invariant functional ${\bf J}_G \colon \prod_i (\cE^1_i \cap \AM_i^{-1}(0)) \to \R$ is defined as
\[
{\bf J}_G(\bphi):=\inf_{g \in G} {\bf J}(g. \bphi).
\]
The functional ${\bf J}_G$ plays a role of an exhaustion function on $\big( \prod_i \tilde{\cH}_i \big)/G$.
\begin{dfn}
We say that the Ding functional ${\bf D}$ is ${\bf J}_G$-coercive if there exists some constants $\d, C>0$ such that
\[
{\bf D}(\bphi) \geq \d {\bf J}_G(\bphi)-C, \quad \bphi \in \prod_i \tilde{\cH}_i.
\]
\end{dfn}
Since each $(\cH_i,d_{1,i})$ is a metric space, the space $\prod_i \cH_i$ has a product metric
\[
{\bf d}_1(\bphi,\bpsi):=\sum_i d_{1,i}(\phi_i,\psi_i).
\]
We also set
\[
{\bf d}_{1,G}(\bphi,\bpsi):=\inf_{f,g \in G} {\bf d}_1(f. \bphi, g. \bpsi), \quad \bphi, \bpsi \in \prod_i \tilde{\cH}_i.
\]
Now we are ready to prove Theorem \ref{Coercivity}. The proof contains some overlap with the K\"ahler-Einstein case \cite[Theorem 7.1]{DR15}.
\begin{proof}[Proof of Theorem \ref{Coercivity}]
We want to apply the general theory for coercivity estimates established in \cite[Theorem 3.4]{DR15} to the quadruple $(\prod_i \tilde{\cH}_i, {\bf d}_1, {\bf D}, G)$. So we have to check that the data $(\prod_i \tilde{\cH}_i, {\bf d}_1, {\bf D}, G)$ satisfies (P1)--(P7) in \cite[Hypothesis 3.2]{DR15}. Let $\cM$ be the set of minimizers of ${\bf D}$ on $\prod_i (\cE_i^1 \cap \AM^{-1}_i(0))$.

\vspace{4mm}

(P1) \emph{For any $\bphi(0), \bphi(1) \in \prod_i \tilde{\cH}_i$ there exists a ${\bf d}_1$-geodesic segment $[0,1] \mapsto \bphi(t) \in \prod_i (\cE_i^1 \cap \AM^{-1}_i(0))$ for which $t \mapsto {\bf D}(\bphi(t))$ is continuous and convex on $[0,1]$.}

\vspace{2mm}

For any $\bphi(0), \bphi(1) \in \prod_i \tilde{\cH}_i$, we take a $k$-tuple of $C^{1,1}$-geodesic segments $\bphi(t):=(\bphi_i(t))$ connecting $\bphi(0)$ with $\bphi(1)$. Clearly, the segment $\bphi(t)$ satisfies the midpoint property, and hence is ${\bf d}_1$-geodesic in $\prod_i (\cE_i^1 \cap \AM^{-1}_i(0))$. Moreover, since $\sum_i \phi_i(t)$ is a $\PSH(\sum_i \omega_{0,i})$-subgeodesic, we can apply Berndtsson's convexity theorem \cite[Theorem 1.1]{Ber15} to conclude that $t \mapsto {\bf D}(\bphi(t))$ is continuous and convex on $[0,1]$.

\vspace{4mm}

(P2) \emph{If $\bphi_m=(\phi_{m,i}) \in \prod_i (\cE^1_i \cap \AM_i^{-1}(0))$ satisfies
\[
\lim_{m \to \infty} {\bf D}(\bphi_m)=\inf_{\prod_i (\cE^1_i \cap \AM_i^{-1}(0))} {\bf D}, \quad {\bf d}_1({\bf 0}, \bphi_m) \leq C,
\]
then there exists an element $\bphi \in \cM$ and a subsequence $\{\bphi_{m_r}\}_r$ ${\bf d}_1$-converging to $\bphi$.}

\vspace{2mm}

Since ${\bf d}_1({\bf 0}, \bphi_m) \leq C$, it follows from \cite[Corollary 5.8]{Dar14} and (27) in \cite{DR15} that $|\sup_X \phi_{m,i}| \leq C$. So by taking a subsequence, we have $\bphi_m \to \bphi=(\phi_i) \in  \prod_i \PSH(X,\omega_{0,i})$ in $L^1$ (\cf \cite[Proposition I.4.21]{Dem12}). Since $\AM_i$ is upper semi-continuous, we have
\[
0=\limsup_{m \to \infty} \AM_i(\phi_{m,i}) \leq \AM_i(\phi_i) \leq [\Omega_i]^n \cdot \limsup_{m \to \infty} \sup_X \phi_{m,i} \leq C.
\]
This shows that $\bphi \in \prod_i \cE^1_i$. Next we show that ${\bf L}(\bphi_m) \to {\bf L}(\bphi)$. Applying $|e^a-e^b| \leq |a-b|(e^a+e^b)$ (for $a, b \in \R$) and the H\"older inequality, we have
\begin{eqnarray*}
\bigg| \int_X (e^{-\sum_i \phi_{m,i}}-e^{-\sum_i \phi_i})\theta_0^n \bigg| &\leq& \int_X \bigg|\sum_i \phi_{m,i}-\sum_i \phi_i \bigg|(e^{-\sum_j \phi_{m, j}}+e^{-\sum_j \phi_j})\theta_0^n \\
&\leq& \int_X \sum_i |\phi_{m,i}-\phi_i|(e^{-\sum_j \phi_{m,j}}+e^{-\sum_j \phi_j}) \theta_0^n\\
&\leq& \sum_i \| \phi_{m,i}-\phi_i \|_{L^2(X,\theta_0^n)} (\| e^{-\sum_j \phi_{m, j}}\|_{L^2(X,\theta_0^n)}\\
&+& \| e^{-\sum_j \phi_j}\|_{L^2(X,\theta_0^n)}).
\end{eqnarray*}
Then the fundamental inequality $|a-b|^2 \leq C|e^a-e^b|$ (for $a,b \geq 0$) together with the semi-continuity theorem \cite{DK01} yields that the first term converges to zero. On the other hand, since $\bphi_m, \bphi \in \prod_i \cE^1_i$, $\AM_i(\phi_{m,i})=0$ and $|\sup_X \phi_{m,i}| \leq C$, Skoda's (uniform) integrability theorem \cite{Zer01} assures that
\[
\int_X e^{-p\phi_{m,i}} \theta_0^n \leq C, \quad \int_X e^{-p\phi_i} \theta_0^n \leq C
\]
for all $i$ and $p>0$. Thus by the H\"older inequality, we find that $\| e^{-\sum_j \phi_{m,j}}\|_{L^2(X,\theta_0^n)}$ and $\| e^{-\sum_j \phi_j}\|_{L^2(X,\theta_0^n)}$ are uniformly bounded. Eventually, we obtain
\[
\lim_{m \to \infty} {\bf D}(\bphi_m) \geq -\sum_i \frac{1}{[\Omega_i]^n} \limsup_{m \to \infty} \AM_i (\phi_{m, i})+{\bf L}(\bphi) \geq {\bf D}(\bphi).
\]
This shows that $\bphi \in \prod_i \cE^1_i$ is a minimizer of ${\bf D}$. Finally, since the equality holds in the above, we have
\[
\sum_i \frac{1}{[\Omega_i]^n} \limsup_{m \to \infty} \AM_i (\phi_{m, i})=\sum_i \frac{1}{[\Omega_i]^n} \AM_i (\phi_i).
\]
By using the upper semi-continuity $\limsup_{m \to \infty} \AM_i(\phi_{m,i}) \leq \AM_i(\phi_i)$ again, we have $\limsup_{m \to \infty} \AM_i (\phi_{m, i})=\AM_i (\phi_i)$ for each $i$. Taking a further subsequence, we obtain $\lim_{m \to \infty} \AM_i(\phi_{m,i})=\AM_i(\phi_i)$. This together with the $L^1$-convergence $\bphi_m \to \bphi$ and \cite[Lemma 5.1 (i)]{DR15} yields that ${\bf d}_1(\bphi_m, \bphi) \to 0$ and $\bphi \in \prod_i (\cE^1_i \cap \AM^{-1}_i(0))$ as desired.

\vspace{4mm}

(P3) \emph{$\cM \subset \prod_i \tilde{\cH}_i$.}

\vspace{2mm}

This is \cite[Theorem 2.9]{HN17}.

\vspace{4mm}

(P4) \emph{$G$ acts on $\prod_i \tilde{\cH}_i$ by ${\bf d}_1$-isometries.}

\vspace{2mm}

This follows from \cite[Lemma 5.9]{DR15}.

\vspace{4mm}

(P5) \emph{$G$ acts on $\cM$ transitively.}

\vspace{2mm}

This is \cite[Theorem 1.5]{HN17}.

\vspace{4mm}

(P6) \emph{If $\cM \neq \emptyset$, then for any $\bphi, \bpsi \in \prod_i \tilde{\cH}_i$ there exists $g \in G$ such that ${\bf d}_{1,G}(\bphi,\bpsi)={\bf d}_1(\bphi, g. \bpsi)$.}

\vspace{2mm}

Suppose $\bphi=(\phi_i) \in \prod_i \tilde{\cH}_i$ is a CKE metric. Thanks to Lemma \ref{MTD}, if we set $K:=\Iso(X,\omega_{\phi_1})=\ldots=\Iso(X,\omega_{\phi_k})$, then the pair $(G,K)$ satisfies the assumption in \cite[Proposition 6.2]{DR15}, which is needed to apply \cite[Proposition 6.8]{DR15}. In addition to this, we have to check the following conditions (i)-(iii):

(i) By the definition of $K$, we have $K. \bphi=\bphi$.

(ii) For each $V \in \Lie(K)$ and geodesic ray $\exp(tJV). \phi_i$, we can use the same argument as in the proof of \cite[Theorem 7.1]{DR15}, and show that $\exp(tJV). \bphi$ is a $d_1$-geodesic.

(iii) For any fixed $\bpsi=(\psi_i), \bvarphi=(\varphi_i) \in \prod_i \tilde{\cH}_i$, the continuity of the map $G \times G \ni (f,g) \mapsto {\bf d}_1(f.\bvarphi, g.\bpsi)$ follows from the continuity of $G \times G \ni (f,g) \mapsto d_{1,i}(f.\varphi_i, g.\psi_i)$ for each $i$ (see the proof of \cite[Theorem 7.1]{DR15}).

\vspace{4mm}

(P7) \emph{For all $\bphi, \bpsi \in \prod_i \tilde{\cH}_i$ and $g \in G$, ${\bf D}(\bpsi)-{\bf D}(\bphi)={\bf D}(g. \bpsi)-{\bf D}(g. \bphi)$.}

\vspace{2mm}

By the definition of ${\bf D}$, the functional ${\bf D}$ can be written as the following path-integral
\[
{\bf D}(\bphi)=\sum_i \frac{1}{[\Omega_i]^n} \int_0^1 \int_X \dot{\phi}_i(t) e^{\rho_i(\bphi(t))} \omega_{\phi_i (t)}^n dt,
\]
where $\bphi \colon [0,1] \to \prod_i \tilde{\cH}_i$ is a smooth path with $\bphi(0)={\bf 0}$ and $\bphi(1)=\bphi$. From the normalization of $\rho_i(\bphi)$, one can easily check that $g^{\ast} \rho_i(\bphi)=\rho_i(g. \bphi)$ for all $g \in G$ and $\bphi \in \prod_i \tilde{\cH}_i$. Thus the functional ${\bf D}$ is a path-integral of the $G$-invariant closed 1-form $\sum_i \frac{1}{[\Omega_i]^n} \int_X \: \cdot \: e^{\rho_i(\bphi)} \omega_{\phi_i}^n$ on $\prod_i \tilde{\cH}_i$. Hence we have
\[
{\bf D}(\bpsi)-{\bf D}(\bphi)={\bf D}(g. \bpsi)-{\bf D}(g. \bphi)
\]
for all $\bphi, \bpsi \in \prod_i \tilde{\cH}_i$ and $g \in G$.

\vspace{4mm}

With all these observations, we have shown that ${\bf D}$ is ${\bf d}_{1,G}$-coercive, \ie there exist constants $\d, C>0$ such that
\[
{\bf D}(\bphi) \geq \d {\bf d}_{1,G}(0,\bphi)-C.
\]
Finally, this is equivalent to ${\bf J}_G$-coercivity by \eqref{Jde}.
\end{proof}
\subsection{Matsushima's type decomposition}
\begin{lem}[See also Corollary 1.6 in \cite{HN17}] \label{MTD}
Let $X$ be a compact K\"ahler manifold with $c_1(X)>0$, and $(\Omega_i)$ a decomposition of $2 \pi c_1(X)$ admitting a coupled K\"ahler-Einstein metric $\bomega=(\omega_i)$. Then we have
\begin{equation} \label{isometry}
K:=\Iso(X,\omega_1)=\ldots=\Iso(X, \omega_k),
\end{equation}
and $G=K^\C$. In particular, $G$ is reductive.
\end{lem}
\begin{proof}
For any $g \in \Iso(X,\omega_1)$. we have
\[
\omega_1+g^\ast \bigg( \sum_{i>1} \omega_i \bigg)=\Ric(g^\ast \omega_1)=\Ric(\omega_1)=\omega_1+\sum_{i>1} \omega_i.
\]
This shows that $g^\ast \bigg( \sum_{i>1} \omega_i \bigg)=\sum_{i>1} \omega_i $. Thus
\[
\Ric(g^\ast \omega_2)=g^\ast \omega_1+g^\ast \bigg( \sum_{i>1} \omega_i \bigg)=\sum_i \omega_i=\Ric(\omega_2).
\]
From the uniqueness of Calabi-Yau solution \cite{Yau78}, we know that $g^\ast \omega_2=\omega_2$ and hence $g \in \Iso(X,\omega_2)$. Similarly, we can show that all $\Iso(X,\omega_i)$'s coincide.

For any $g \in G$, $g^\ast \bomega$ is also CKE. By \cite[Theorem 1.5]{HN17} and the proof therein, there exists a real holomorphic vector field $V$ which generates a one-parameter subgroup $f_t$ such that $g^\ast \bomega=f_1^\ast \bomega$. We set $f_t^\ast \bomega=\bomega_0+\dd \bphi(t)$ for $\bphi \colon [0,1] \to \prod_i \cH_i$. Since $H^{0,1}(X)=0$, we have $i_V \bomega=\sqrt{-1} \bp \dot{\bphi}(t)$. This shows that $JV$ is $\omega_i$-Killing for all $i=1,\ldots,k$. Hence $f_t \in K^\C$. On the other hand, since $f_1 \circ g \in K$, we obtain $g=f_1^{-1} \circ f_1  \circ g \in K^\C$ as desired.
\end{proof}
\subsection{Twisted K\"ahler-Einstein metrics}
In this subsection, let $X$ be a compact K\"ahler manifold and $\a_0$ a K\"ahler form with $\lambda c_1(X)-[\a_0]>0$. We take a reference K\"ahler metric $\omega_0 \in \lambda c_1(X)-[\a_0]$ and set
\[
\cH_{\omega_0}:=\{\phi \in \PSH(X,\omega_0)|\omega_\phi:=\omega_0+\dd \phi>0\},
\]
\[
\cH_{\a_0}:=\{\psi \in \PSH(X,\a_0)|\a_\psi:=\a_0+\dd \psi>0\}.
\]
\begin{lem} \label{TKE}
Assume that there exists a pair $(\phi,\psi) \in \cH_{\omega_0} \times \cH_{\a_0}$ satisfying the $\a_\psi$-twisted K\"ahler-Einstein equation:
\begin{equation} \label{TKE}
\Ric(\omega_\phi)=\lambda(\omega_\phi+\a_\psi).
\end{equation}
Then for any $\tilde{\psi} \in \cH_{\a_0}$, there exists a unique $\a_{\tilde \psi}$-TKE potential $\tilde{\phi} \in \cH_{\omega_0}$ up to additive constants. In particular, the twisted inverse Ricci operator
\[
(\Ric-\lambda \Id)^{-1} \colon \cH_{\a_0} \cap \AM^{-1}(0) \to \cH_{\omega_0} \cap \AM^{-1}(0)
\]
given by $\tilde{\psi} \mapsto \tilde{\phi}$ is well-defined.
\end{lem}
\begin{proof}
In the $\lambda=-1$ case, one can show the existence as well as the uniqueness with a slight modification to the original arguments \cite{Aub76}, so we omit the proof. In the $\lambda=1$ case, the uniqueness follows from \cite[Theorem 7.1]{Ber15} since the strict positivity of $\a_\psi$ yields that any holomorphic vector field $V$ satisfying $i_V \a_\psi=0$ must be zero. As for the existence, one can show using continuity method along the path $\{\phi_t\} \subset \cH_{\omega_0}$:
\begin{equation} \label{cmethod}
\Ric(\omega_{\phi_t})-\a_{\tilde{\psi}}=(1-t)\omega_0+t\omega_{\phi_t}, \quad t \in [0,1].
\end{equation}
We can solve \eqref{cmethod} for $t=0$ due to the Calabi-Yau theorem \cite{Yau78}, while for $t \in (0,1]$, the solvability of \eqref{cmethod} depends only on the classes $([\omega_0],[\a_0])$, which was observed in \cite{Sze11} when $\a_{\tilde{\psi}}=0$, but the proof carries over essentially verbatim to the general case.
\end{proof}
\begin{rem}
For any $\psi \in \cH_{\a_0}$, $\a_\psi$-TKE metrics are strictly unique, whereas the CKE metrics are not when $G$ is non-trivial. Even in this case, the above lemma assures that the coupled Ricci iteration \eqref{CRI} is well-defined.
\end{rem}
\section{Coupled Ricci iteration} \label{CIB}
\subsection{General case}
Let $X$ be a compact K\"ahler manifold and $(\Omega_i)$ a decomposition of $2 \pi \lambda c_1(X)$. In what follows, we always assume that the Ding functional ${\bf D}$ is ${\bf J}$-coercive (in particular, which forces $G$ to be trivial). We note that in the $\lambda=-1$ case, the coercivity of ${\bf D}$ was essentially shown in \cite{HN17}.
We start with the following lemma:
\begin{lem} \label{continuity of the iteration map}
The iteration map $\cI \colon \prod_i \tilde{\cH}_i \to \prod_i \tilde{\cH}_i$ admits a unique ${\bf d}_1$-continuous extension $\cI \colon \prod_i(\cE^1_i \cap \AM^{-1}_i(0)) \to \prod_i (\cE^1_i \cap \AM^{-1}_i(0))$.
\end{lem}
\begin{proof}
For $\bphi=(\phi_i) \in \prod_i \tilde{\cH}_i$, we set $\bpsi:=\cI(\bphi)=(\psi_i) \in \prod_i \tilde{\cH}_i$, \ie $\psi_i$ satisfies
\begin{eqnarray*}
\Ric(\omega_{\psi_i}) &=& \lambda \bigg( \sum_{j \leq i} \omega_{\psi_j}+\sum_{j>i} \omega_{\phi_j} \bigg)\\
&=& \lambda \bigg( \sum_j \omega_{0,j}+\dd \bigg( \sum_{j \leq i} \psi_j+\sum_{j>i} \phi_j \bigg) \bigg)\\
&=& \Ric(\theta_0)+\lambda \dd \bigg( \sum_{j \leq i} \psi_j+\sum_{j>i} \phi_j \bigg).
\end{eqnarray*}
Since $\Ric(\omega_{\psi_i})-\Ric(\theta_0)=-\dd \log(\omega_{\psi}^n/\theta_0^n)$, we find that $\bpsi$ satisfies the following coupled Monge-Amp\`ere equation
\begin{equation} \label{IMA}
\frac{\omega_{\psi_i}^n}{[\Omega_i]^n}=\frac{e^{-\lambda(\sum_{j \leq i} \psi_j+\sum_{j>i} \phi_j)} \theta_0^n}{\int_X e^{-\lambda(\sum_{j \leq i} \psi_j+\sum_{j>i} \phi_j)} \theta_0^n}.
\end{equation}
The equation \eqref{IMA} can be extended for all elements $\bphi, \bpsi \in \prod_i (\cE^1_i \cap \AM^{-1}_i(0))$ in a weak sense (as an equation for measures), and each solution is characterized as a unique minimizer of the Ding functional restricted to a slice of $\prod_i (\cE^1_i \cap \AM^{-1}_i(0))$:
\[
{\bf D}_i \colon \prod_{j<i} \{\psi_j\} \times (\cE^1_i \cap \AM^{-1}_i(0)) \times \prod_{j>i} \{ \phi_j \} \to \R, \quad i=1,\ldots,k.
\]
Indeed, we know that each ${\bf D}_i$ is $J_i$-coercive since ${\bf D}$ is ${\bf J}$-coercive.  So we can construct a minimizer $\psi_i \in \cE^1_i \cap \AM^{-1}(0)$ as a limit point of any minimizing sequence of ${\bf D}_i$ since ${\bf D}_i$ is lower semi-continuous and sublevel sets of $J_i$ are weakly compact (\cf \cite[Lemma 3.3]{BBGZ13}). As for the uniqueness in the Fano case, we first see that the RHS of \eqref{IMA} is in $L^p$ for $p>1$ by Skoda's integrability theorem. Thus any solution $\psi_i$ to \eqref{IMA} is bounded due to the Ko\l odziej's $C^0$-estimate \cite{Kol03, Kol05} (or the Alexandroff-Bakelman-Pucci estimate \cite{Blo05}). Then any two solutions $\psi_i$, $\tilde{\psi}_i$ can be joined by a unique weak geodesic (or more pricisely, bounded geodesic) $\psi_i(t)$ with $\psi_i(0)=\psi_i$ and $\psi_i(1)=\tilde{\psi}_i$ constructed by the Peron-Bremermann type envelope (\cf \cite[Section 2.2]{Ber15}). Since $\psi_i$ and $\tilde{\psi_i}$ are a minimizer of ${\bf D}_i$ and ${\bf D}_i$ is convex along $\psi_i(t)$, ${\bf D}_i(\psi_i(t))$ should be affine in $t$. Moreover, since $e^{-(\sum_{j<i} \psi_j+\sum_{j>i} \phi_j)}$ is integrable, we can apply \cite[Theorem 7.1]{Ber15} to know that $\psi_i(t)=g_t.\psi_i$ for some one-parameter subgroup $g_t$ generated by a holomorphic vector field, which should be identity from the assumption that $G$ is trivial.

Then one can prove in the same way as the proof of Theorem \ref{Coercivity} (P2). Let us consider a ${\bf d}_1$-convergent sequence $\bphi_m=(\phi_{m,i}) \to \bphi_{\infty}:=\bphi$ in $\prod_i( \cE^1_i \cap \AM^{-1}_i(0))$ as $m \to \infty$. Then each $\bphi_m$ defines a slice of $\prod_i (\cE^1_i \cap \AM^{-1}_i(0))$:
\[
L_m:=\big\{(\varphi,\phi_{m,2},\ldots,\phi_{m,k}) \in\prod_i( \cE^1_i \cap \AM^{-1}_i(0)) \big| \varphi \in \cE^1_1 \cap \AM^{-1}_1(0)\big\} \simeq \cE^1_1 \cap \AM^{-1}_1(0).
\]
Let $\varphi_m$ (resp. $\psi_1$) be a unique minimizer of ${\bf D}_{| L_m}$ (resp. ${\bf D}_{| L_\infty}$) on $\cE^1_1 \cap \AM^{-1}_1(0)$. So we have
\[
{\bf D}(\varphi_m, \phi_{m,2},\ldots,\phi_{m,k}) \leq {\bf D}(\varphi, \phi_{m,2},\ldots,\phi_{m,k})
\]
for all $\varphi \in \cE^1_1$. In particular, we know that ${\bf D}(\varphi_m, \phi_{m,2},\ldots,\phi_{m,k}) \leq C$, and hence $J_1(\varphi_m) \leq C$ since ${\bf D}$ is ${\bf J}$-coercive (in fact, the coercivity estimate holds on the ${\bf d}_1$-metric completion $\prod_i (\cE^1_i \cap \AM^{-1}_i(0))$ by density \cite[Remark 3.9 (ii)]{DR15}). Thus we also have $d_{1,1}(0,\varphi_m) \leq C$ by \eqref{Jde}, and $|\sup_X \varphi_m| \leq C$. So by taking a subsequence, we obtain $\varphi_m \to \varphi_\infty \in \cE^1_1$ in $L^1$. The lower semi-continuity of ${\bf D}$ yields that
\begin{eqnarray} \label{mcm}
{\bf D}(\varphi_\infty,\phi_2,\ldots,\phi_k) &\leq& \liminf_{m \to \infty} {\bf D}(\varphi_m,\phi_{m,2},\ldots,\phi_{m,k}) \nonumber \\
&\leq& \liminf_{m \to \infty}{\bf D}(\varphi,\phi_{m,2},\ldots,\phi_{m,k}) \nonumber \\
&=& {\bf D}(\varphi,\phi_2,\ldots,\phi_k)
\end{eqnarray}
for all $\varphi \in \cE^1_1$. Put $\varphi=\varphi_\infty$, then we find that all of the above inequalities must be equalities. In particular,
\[
\limsup_{m \to \infty} \AM_1(\varphi_m)=\AM_1(\varphi_\infty).
\]
By passing to a subsequence, we get $\lim_{m \to \infty} \AM_1(\varphi_m)=\AM_1(\varphi_\infty)$, $d_{1,1}(\varphi_m,\varphi_\infty) \to 0$ and hence $\AM_1(\varphi_\infty)=0$. Next we put $\varphi=\psi_1$ in \eqref{mcm} to know that $\varphi_\infty=\psi_1$ since $\psi_1$ is a unique minimizer of ${\bf D}_{L_\infty}$ with $\AM_1(\psi_1)=0$, which also shows that convergence $d_{1,1}(\varphi_m,\psi_1) \to 0$ holds without taking subsequences. Thus we find that the unique minimizer $\psi_1$ of ${\bf D}_1$ depends ${\bf d}_1$-continuously on $\bphi$. 

Similarly, we can check that the unique minimizer $\psi_2$ of ${\bf D}_2$ depends ${\bf d}_1$-continuously on $\bphi$. Repeating this argument, we find that the $k$-tuple $\bpsi=\cI(\bphi)$  depends ${\bf d}_1$-continuously on $\bphi$.
\end{proof}
\begin{rem}
For the later purpose, it is enough to show the continuity of the map $\cI$ with respect to the $C^{r,\b}$-topology ($\b \in (0,1)$), which easily follows from the Implicit Function Theorem. However, we believe Lemma \ref{continuity of the iteration map} is of independent interest. Indeed, Lemma \ref{continuity of the iteration map} (together with Proposition \ref{gpi}) suggests that one can extend \eqref{CRI} to a system defined on the larger space $\prod_i (\cE_i^1 \cap \AM^{-1}_i(0))$, and ask if this possibly singular system converges in the ${\bf d}_1$-topology. For the time $1$ Ricci iteration, such kind of extension was studied in \cite[Theorem 6.4]{BBEGZ11}.
\end{rem}
Next we prove the following proposition, which plays a key role to prove convergence of the coupled Ricci iteration:
\begin{prop} \label{gpi}
The following statements hold:
\begin{enumerate}
\item (Monotonicity) The Ding functional ${\bf D} \colon \prod_i(\cE^1_i \cap \AM^{-1}_i(0))  \to \R$ is decreasing along the coupled Ricci iteration, \ie we have
\begin{equation} \label{MAI}
{\bf D}(\cI(\bphi)) \leq {\bf D}(\bphi)
\end{equation}
for all $\bphi \in \prod_i(\cE^1_i \cap \AM^{-1}_i(0))$.
\item (Strict monotonicity) The equality in \eqref{MAI} holds if and only if $\bphi$ is CKE.
\end{enumerate}
\end{prop}
\begin{proof}
(1) We carry over the same notations from the proof of Lemma \ref{continuity of the iteration map}. Since $\psi_1$ is the minimizer of ${\bf D}_1$, we have ${\bf D}(\psi_1,\phi_2,\ldots,\phi_k) \leq {\bf D}(\bphi)$. By repeating this, we get
\begin{eqnarray} \label{imo}
{\bf D}(\bphi) &\geq& {\bf D}(\psi_1,\phi_2,\ldots,\phi_k) \nonumber \\
&\geq& \hdots \nonumber \\
&\geq& {\bf D}(\psi_1,\ldots,\psi_{k-1},\phi_k) \nonumber \\
&\geq& {\bf D}(\bpsi).
\end{eqnarray}

(2) The ``if part'' is trivial since CKE metrics are exactly fixed points of the iteration map $\cI$. Conversely, assume ${\bf D}(\bpsi) = {\bf D}(\bphi)$. Then all the inequalities in \eqref{imo} must be equalities. From the uniquness of the minimizer of ${\bf D}_i$ in each step, $\bpsi$ should be equal to $\bphi$, thus $\bphi$ satisfies
\[
\frac{\omega_{\phi_1}^n}{[\Omega_1]^n}=\ldots=\frac{\omega_{\phi_k}^n}{[\Omega_k]^n}=\frac{e^{-\lambda \sum_i \phi_i} \theta_0^n}{\int_X e^{-\lambda \sum_i \phi_i} \theta_0^n}
\]
in a weak sense. By applying regularity theorem for CKE metrics \cite[Theorem 2.9]{HN17}, we see that $\bphi$ is CKE in the classical sense.
\end{proof}
For any initial data $\bomega_0=(\omega_{0,i})$ with $\omega_{0,i} \in \Omega_i$, we consider the coupled Ricci iteration $\bomega_\ell=(\omega_{\ell,i})=\cI^\ell(\bomega_0)$. Define potential functions $\phi_{\ell,i}$ by
\[
\omega_{\ell,i}=\omega_{0,i}+\dd \phi_{\ell,i}.
\]
Normalizing by constants, the function $\phi_{\ell+1,i}$ satisfies the following sup-normalized coupled Monge-Amp\`ere equations:
\begin{equation} \label{NMA}
\frac{\omega_{\phi_{\ell+1,i}}^n}{[\Omega_i]^n}=\frac{e^{-\lambda(\sum_{j \leq i} \phi_{\ell+1,j}+\sum_{j>i} \phi_{\ell,j})} \theta_0^n}{\int_X e^{-\lambda(\sum_{j \leq i} \phi_{\ell+1,j}+\sum_{j>i} \phi_{\ell,j})} \theta_0^n}, \quad \sup_X \phi_{\ell,i}=0.
\end{equation}
\begin{lem} \label{HOE}
Let $\phi_{\ell,i}$ be the solution of \eqref{NMA} and assume  a uniform bound $\|\phi_{\ell,i}\|_{C^0} \leq C$. Then we have $\|\phi_{\ell,i}\|_{C^r} \leq C$ for all $r \geq 0$. Moreover, we have the uniform equivalence of K\"ahler metrics
\begin{equation} \label{ueq}
\frac{1}{C} \omega_{0,i} \leq \omega_{\ell,i} \leq C \omega_{0,i}.
\end{equation}
\end{lem}
\begin{proof}
\textbf{(Uniform Laplacian estimates)}

From the uniform bound $\|\phi_{\ell,i} \|_{C^0} \leq C$ and \eqref{NMA}, we obtain
\[
\frac{1}{C} \omega_{0,i}^n \leq \omega_{\ell+1,i}^n \leq C \omega_{0,i}^n.
\]
Combining with the fundamental inequality
\[
\Tr_{\omega_{0,i}} \omega_{\ell+1,i} \leq \frac{1}{(n-1)!}(\Tr_{\omega_{\ell+1,i}} \omega_{0,i})^{n-1} \frac{\omega_{\ell+1,i}^n}{\omega_{0,i}^n},
\]
we have
\begin{equation} \label{FIE}
\Tr_{\omega_{0,i}} \omega_{\ell+1,i} \leq C (\Tr_{\omega_{\ell+1,i}} \omega_{0,i})^{n-1}.
\end{equation}
A standard computation shows that
\[
\Delta_{\ell+1,i} \log \Tr_{\omega_{0,i}} \omega_{\ell+1,i} \geq \frac{1}{\Tr_{\omega_{0,i}} \omega_{\ell+1,i}} \bigg[ \Delta_{0,i} \log \frac{\omega_{\ell+1,i}^n}{\omega_{0,i}^n}-R_{0,i} \bigg]-C_1 \Tr_{\omega_{\ell+1,i}} \omega_{0,i},
\]
where $R_{0,i}$ is the scalar curvature of $\omega_{0,i}$, and $\Delta_{0,i}$ (resp. $\Delta_{\ell+1,i}$) is the $\bp$-Laplacian with respect to $\omega_{0,i}$ (resp. $\omega_{\ell+1,i}$). A constant $C_1$ only depends on a lower bound for the bisectional curvature of the reference metric $\omega_{0,i}$ (for instance, see \cite[Section 7]{Rub14}). Using the equation \eqref{NMA}, we obtain
\begin{eqnarray*}
\Delta_{\ell+1,i} \log \Tr_{\omega_{0,i}} \omega_{\ell+1,i} &\geq& \frac{1}{\Tr_{\omega_{0,i}} \omega_{\ell+1,i}} \bigg[ -\lambda \Delta_{0,i} \bigg(\sum_{j \leq i} \phi_{\ell+1,j}+\sum_{j>i} \phi_{\ell,j} \bigg)-C_2 \bigg] \\
&-& C_1 \Tr_{\omega_{\ell+1,i}} \omega_{0,i}.
\end{eqnarray*}
Now we deal with the two cases $\lambda=\pm 1$ separately.

\vspace{2mm}

(1) $\lambda=-1$ case. By using $\Tr_{\omega_{0,i}} \omega_{\ell+1,i}=n+\Delta_{0,i} \phi_{\ell+1,i}$, $n+\Delta_{0,i} \phi_{\ell+1,j}>0$ and $n+\Delta_{0,i} \phi_{\ell,j}>0$, the above can be reduced to
\[
\Delta_{\ell+1,i} \log \Tr_{\omega_{0,i}} \omega_{\ell+1,i} \geq 1-\frac{C_3}{\Tr_{\omega_{0,i}} \omega_{\ell+1,i}}-C_1 \Tr_{\omega_{\ell+1,i}} \omega_{0,i}.
\]
Thus subtracting $\Delta_{\ell+1,i} \phi_{\ell+1,i}=n-\Tr_{\omega_{\ell+1,i}} \omega_{0,i}$ from the both sides and using the Cauchy-Schwarz inequality
\[
n \leq \Tr_{\omega_{0,i}} \omega_{\ell+1,i} \cdot \Tr_{\omega_{\ell+1,i}} \omega_{0,i},
\]
we have
\[
\Delta_{\ell+1,i}(\log \Tr_{\omega_{0,i}} \omega_{\ell+1,i}-A \phi_{\ell+1,i}) \geq \bigg(A-C_1-\frac{C_3}{n} \bigg) \Tr_{\omega_{\ell+1,i}} \omega_{0,i}-An
\]
for a large constant $A>C_1+\frac{C_3}{n}$. Then the maximum principle shows that at the maximum point $\hat{x} \in X$ of the function $\log \Tr_{\omega_{0,i}} \omega_{\ell+1,i}-A \phi_{\ell+1,i}$, we have $\Tr_{\omega_{\ell+1,i}} \omega_{0,i} (\hat{x}) \leq C_4$. The uniform bound of $|\phi_{\ell+1,i}| \leq C$ together with  \eqref{FIE} implies that $\Tr_{\omega_{0,i}} \omega_{\ell+1,i} \leq C_5$ on $X$. By switching the role of $\omega_{0,i}$ and $\omega_{\ell+1,i}$ in \eqref{FIE}, we obtain also $\Tr_{\omega_{\ell+1,i}} \omega_{0,i} \leq C_6$.

\vspace{2mm}

(2) $\lambda=1$ case. Take a constant $B>0$ so that $B \omega_{0,i} \geq \sum_i \omega_{0,i}$, and put $\Phi:=\sum_{j \leq i} \phi_{\ell+1,j}+\sum_{j>i} \phi_{\ell,j}$ to simplify notations. Since $\Phi$ is $B \omega_{0,i}$-PSH, we observe that
\begin{eqnarray*}
0 &\leq& B \omega_{0,i}+\dd \Phi \\
&\leq& \Tr_{\omega_{\ell+1,i}} (B \omega_{0,i}+\dd \Phi) \cdot \omega_{\ell+1,i} \\
&=& (B \Tr_{\omega_{\ell+1,i}} \omega_{0,i}+\Delta_{\ell+1,i} \Phi) \cdot \omega_{\ell+1,i}.
\end{eqnarray*}
Taking the trace with respect to $\omega_{0,i}$ implies that
\[
\frac{\Delta_{0,i} \Phi}{\Tr_{\omega_{0,i}} \omega_{\ell+1,i}} \leq B \Tr_{\omega_{\ell+1,i}} \omega_{0,i}+\Delta_{\ell+1,i} \Phi-\frac{Bn}{\Tr_{\omega_{0,i}} \omega_{\ell+1,i}}.
\]
So by the Cauchy-Schwarz inequality, we can deduce that
\[
\Delta_{\ell+1,i} (\log \Tr_{\omega_{0,i}} \omega_{\ell+1,i}+\Phi) \geq -C_7 \Tr_{\omega_{\ell+1,i}} \omega_{0,i}.
\]
Subtracting $\Delta_{\ell+1,i} \phi_{\ell+1,i}=n-\Tr_{\omega_{\ell+1,i}} \omega_{0,i}$, we obtain
\[
\Delta_{\ell+1,i} (\log \Tr_{\omega_{0,i}} \omega_{\ell+1,i}+\Phi-K \phi_{\ell+1,i}) \geq (K-C_7) \Tr_{\omega_{\ell+1,i}} \omega_{0,i}-Kn.
\]
for a large constant $K>C_7$. Since $\Phi$ and $\phi_{\ell+1,i}$ are uniformly controlled from the assumption, we apply the maximum principle to conclude that $\Tr_{\omega_{0,i}} \omega_{\ell+1,i} \leq C_8$, and hence $\Tr_{\omega_{\ell+1,i}} \omega_{0,i} \leq C_9$.

\vspace{4mm}

\textbf{(Higher order estimates)}
The above estimates imply that the coupled Monge-Amp\`ere equations \eqref{NMA} are uniformly elliptic. We therefore apply the Evans-Krylov estimate (\cf \cite{Kry82, Wan12}) and obtain $\|\phi_{\ell,i}\|_{C^{2,\b}} \leq C$ for each $\b \in (0,1)$. Combining with the Schauder theory and a standard bootstrapping argument, we obtain the higher order estimates $\|\phi_{\ell,i}\|_{C^r} \leq C$ for all $r \geq 0$. This completes the proof.
\end{proof}
\subsection{For negative first Chern class}
Let $X$ be a compact K\"ahler manifold with $c_1(X)<0$. We start with the following a priori estimate:
\begin{lem} \label{unc}
In the $\lambda=-1$ case, we have a uniform bound $\|\phi_{\ell,i}\|_{C^0} \leq C$ for all $i=1,\ldots,k$.
\end{lem}
\begin{proof}
We first make a few adjustment to $\phi_{\ell,i}$. Define a function $\psi_{\ell,i}$ so that
\[
\omega_{\ell,i}=\omega_{0,i}+\dd \psi_{\ell,i},
\]
and
\begin{equation} \label{UNM}
\frac{\omega_{\psi_{\ell+1,i}}^n}{[\Omega_i]^n}=e^{\sum_{j \leq i} \psi_{\ell+1,j}+\sum_{j>i} \psi_{\ell,j}} \cdot \frac{\theta_0^n}{[\theta_0]^n}, \quad \psi_{0,1}=\ldots=\psi_{0,k}=0.
\end{equation}
By the uniquness of solutions (\cf Lemma \ref{TKE}), two functions $\phi_{\ell+1,i}$, $\psi_{\ell+1,i}$ differ by a constant. Since the function $\sum_{j \leq i} \psi_{\ell+1,j}+\sum_{j>i} \psi_{\ell,j}$ is $\sum_i \omega_{0,i}$-PSH and $C \theta_0 \geq \sum_i \omega_{0,i}$ for some $C>0$, this is also $C \theta_0$-PSH. Thus a Green function estimate together with Jensen's inequality implies that
\begin{eqnarray*}
\sum_{j \leq i} \psi_{\ell+1,j}+\sum_{j>i} \psi_{\ell,j} &\leq& \frac{1}{[\theta_0]^n} \int_X \bigg( \sum_{j \leq i} \psi_{\ell+1,j}+\sum_{j>i} \psi_{\ell,j} \bigg) \theta_0^n+C \\
&\leq& \log \int_X e^{\sum_{j \leq i} \psi_{\ell+1,j}+\sum_{j>i} \psi_{\ell,j}} \cdot \frac{\theta_0^n}{[\theta_0]^n}+C \\
&\leq& C,
\end{eqnarray*}
where we used \eqref{UNM} in the last inequality. We therefore apply Ko\l odziej's $C^0$-estimate \cite{Kol03, Kol05} to find that
\[
\osc_X \psi_{\ell+1,i} \leq C.
\]
Since $\phi_{\ell+1,i}$ is sup-normalized, we conclude that $\|\phi_{\ell+1,i}\|_{C^0} \leq C$.
\end{proof}
Now we are ready to prove Theorem \ref{NFC}.
\begin{proof}[Proof of Theorem \ref{NFC}]
Since the Ding functional ${\bf D}$ is bounded from below and decreasing along the iteration $\bphi_{\ell}=(\phi_{\ell,i})$ (\cf Proposition \ref{gpi} (1)), we have
\[
\lim_{\ell \to \infty} {\bf D}(\bphi_\ell)=:m
\]
without taking subsequences. On the other hand, the uniform estimates $\| \phi_{\ell, i} \|_{C^r} \leq C$ together with Ascoli-Arzel\`a theorem implies that by taking a subsequence, $\bphi_\ell$ converges to some $k$-tuple of functions $\bphi_{\infty}=(\phi_{\infty,i})$ in the $C^{\infty}$-topology. Moreover, the inequality \eqref{ueq} assures that $\bomega_\infty:=\bomega_0+\dd \bphi_{\infty}>0$. In order to show $\bomega_{\infty}$ is CKE, we consider the metric $\cI(\bomega_\infty)$, which, by continuity of $\cI$ (\cf Lemma \ref{continuity of the iteration map}), also arises as a limit point of the iteration $\bomega_{\ell}$. Thus
\[
{\bf D}(\bomega_{\infty})={\bf D}(\cI(\bomega_{\infty}))=m.
\]
So it follows from Proposition \ref{gpi} (2) that $\bomega_{\infty}$ is CKE. Since the CKE metric is unque, the $C^\infty$-convergence $\bomega_\ell \to \bomega_\infty$ in fact holds without taking a subsequence. This completes the proof.
\end{proof}
\subsection{For positive first Chern class}
Let $X$ be a compact K\"ahler manifold with $c_1(X)>0$. We assume that the decomposition $(\Omega_i)$ admits a CKE metric and $G$ is trivial, which is, by Theorem \ref{Coercivity}, equivalent to the ${\bf J}$-coercivity of the Ding functional
\[
{\bf D}(\bphi) \geq \d {\bf J}(\bphi)-C, \quad \bphi \in \prod_i \tilde{\cH}_i.
\]
In the same way as in the $\lambda=-1$ case, the problem can be reduced to the $C^0$-estimate along the iteration. Indeed, this is possible assuming the existence of CKE metrics:
\begin{lem}
Under the above assumptions, we have a uniform bound $\|\phi_{\ell,i}\|_{C^0} \leq C$.
\end{lem}
\begin{proof}
As in \cite{Rub08}, the idea is to use the uniform positive lower bound of the Ricci curvature along the iteration.
Applying the coercivity estimate to the iteration $\{ \bomega_\ell \}$ implies that
\[
{\bf D}(\bomega_\ell) \geq \d {\bf J}(\bomega_\ell)-C.
\]
From Proposition \ref{gpi} (1), the LHS is uniformly bounded from above. Thus we have the uniform bound ${\bf J}(\bomega_\ell) \leq C$. Now we consider the solution $\bphi_\ell=(\phi_{\ell,i})$ to the sup-normalized Monge-Amp\`ere equation \eqref{NMA}. Let $G_{\ell, i}$ (resp. $\Delta_{\ell,i}$) be the Green function (resp. $\bp$-Laplacian) for $\omega_{\ell,i}$, and set
\[
A_{\ell,i}:=-\inf_{(x,y) \in X \times X, \; x \neq y} G_{\ell,i}(x,y).
\]
Since $n+\Delta_{0,i} \phi_{\ell+1,i}>0$ and $n-\Delta_{\ell+1,i} \phi_{\ell+1,i}>0$, a Green function estimate implies that
\[
\phi_{\ell,i}-\frac{1}{[\Omega_i]^n} \int_X \phi_{\ell,i} \omega_{0,i}^n=-\frac{1}{[\Omega_i]^n} \int_X G_{0,i}(x,y) \Delta_{0,i} \phi_{\ell,i}(y) \omega_{0,i}^n \leq n A_{0,i},
\]
\[
\phi_{\ell,i}-\frac{1}{[\Omega_i]^n} \int_X \phi_{\ell,i} \omega_{\ell,i}^n=-\frac{1}{[\Omega_i]^n} \int_X G_{\ell,i}(x,y) \Delta_{\ell,i} \phi_{\ell,i}(y) \omega_{\ell,i}^n \geq -n A_{\ell,i}.
\]
Hence we have
\[
\osc_X \phi_{\ell,i} \leq n(A_{0,i}+A_{\ell,i})+I_i(\phi_{\ell,i}).
\]
The third term in the RHS is uniformly bounded from above by \eqref{JAI} and $J_i(\omega_{\ell,i}) \leq {\bf J}(\bomega_\ell) \leq C$. As for the second term, we can use \cite[Theorem 3.2]{BM85} and obtain
\[
A_{\ell,i} \leq c(n) \cdot \frac{{\rm diam}(X, \omega_{\ell,i})}{{\rm Vol}(X, \omega_{\ell,i})}
\]
since $\Ric(\omega_{\ell,i})$ is non-negative. Indeed, along the iteration, we have a uniform lower bound for the Ricci curvature $\Ric(\omega_{\ell+1,i})=\sum_{j \leq i} \omega_{\ell+1,j}+\sum_{j>i} \omega_{\ell,j}>\omega_{\ell+1,i}$. Moreover, applying Myers' Theorem, we find that
\[
{\rm diam}(X, \omega_{\ell,i})^2 \leq (2n-1) \pi^2.
\]
Thus we have the uniform upper bound for $A_{\ell,i}$, and consequently we have $\|\phi_{\ell,i}\|_{C^0} \leq C$.
\end{proof}
\begin{proof}[Proof of Theorem \ref{PFC}]
We invoke Lemma \ref{HOE} to obtain the higher order estimates for $\phi_{\ell,i}$ as well as the uniform equivalence of K\"ahler metrics. Along the same line as in the proof of Theorem \ref{NFC}, we can show the $C^\infty$-convergence to the unique CKE metric $\bomega_{\ell} \to \bomega_{\infty}$ as $\ell \to \infty$. This completes the proof of Theorem \ref{PFC}.
\end{proof}

\end{document}